\documentclass[12pt,reqno]{amsart} 
\usepackage{amssymb,amscd,url}

\pdfoutput=1

\begin{document}

\allowdisplaybreaks


\title[Degeneration of Dynamical Degrees]
{Degeneration of Dynamical Degrees in Families of Maps}
\date{\today}
\author[Joseph H. Silverman]{Joseph H. Silverman}
\email{jhs@math.brown.edu}
\address{Mathematics Department, Box 1917
         Brown University, Providence, RI 02912 USA}
\author[Gregory S. Call]{Gregory S. Call}
\email{gscall@amherst.edu}
\address{Department of Mathematics and Statistics, Amherst College,
  Am\-herst, MA 01002 USA}
\thanks{Silverman's research supported by Simons Collaboration Grant
  \#241309\\
  \textbf{Authors' Note}: This is an expanded version of the article
  publishd in \emph{Acta Arithmetica}. It contains a corrected
  statement and full proof of Propostion~\ref{proposition:invim}(c).}



\hyphenation{ca-non-i-cal semi-abel-ian}


\newtheorem{theorem}{Theorem}
\newtheorem{lemma}[theorem]{Lemma}
\newtheorem{conjecture}[theorem]{Conjecture}
\newtheorem{question}[theorem]{Question}
\newtheorem{proposition}[theorem]{Proposition}
\newtheorem{corollary}[theorem]{Corollary}
\newtheorem*{claim}{Claim}

\theoremstyle{definition}
\newtheorem*{definition}{Definition}
\newtheorem{example}[theorem]{Example}
\newtheorem{remark}[theorem]{Remark}
\newtheorem{numberedclaim}[theorem]{Claim}

\theoremstyle{remark}
\newtheorem*{acknowledgement}{Acknowledgements}


\newenvironment{notation}[0]{%
  \begin{list}%
    {}%
    {\setlength{\itemindent}{0pt}
     \setlength{\labelwidth}{4\parindent}
     \setlength{\labelsep}{\parindent}
     \setlength{\leftmargin}{5\parindent}
     \setlength{\itemsep}{0pt}
     }%
   }%
  {\end{list}}

\newenvironment{parts}[0]{%
  \begin{list}{}%
    {\setlength{\itemindent}{0pt}
     \setlength{\labelwidth}{1.5\parindent}
     \setlength{\labelsep}{.5\parindent}
     \setlength{\leftmargin}{2\parindent}
     \setlength{\itemsep}{0pt}
     }%
   }%
  {\end{list}}
\newcommand{\Part}[1]{\item[\upshape#1]}

\def\Case#1#2{%
\paragraph{\textbf{\boldmath Case #1: #2.}}\hfil\break\ignorespaces}

\renewcommand{\a}{\alpha}
\renewcommand{\b}{\beta}
\newcommand{\g}{\gamma}
\renewcommand{\d}{\delta}
\newcommand{\e}{\epsilon}
\newcommand{\f}{\varphi}
\newcommand{\bfphi}{{\boldsymbol{\f}}}
\renewcommand{\l}{\lambda}
\renewcommand{\k}{\kappa}
\newcommand{\lhat}{\hat\lambda}
\newcommand{\m}{\mu}
\newcommand{\bfmu}{{\boldsymbol{\mu}}}
\renewcommand{\o}{\omega}
\renewcommand{\r}{\rho}
\newcommand{\rbar}{{\bar\rho}}
\newcommand{\s}{\sigma}
\newcommand{\sbar}{{\bar\sigma}}
\renewcommand{\t}{\tau}
\newcommand{\z}{\zeta}

\newcommand{\D}{\Delta}
\newcommand{\G}{\Gamma}
\newcommand{\F}{\Phi}
\renewcommand{\L}{\Lambda}

\newcommand{\ga}{{\mathfrak{a}}}
\newcommand{\gb}{{\mathfrak{b}}}
\newcommand{\gn}{{\mathfrak{n}}}
\newcommand{\gp}{{\mathfrak{p}}}
\newcommand{\gP}{{\mathfrak{P}}}
\newcommand{\gq}{{\mathfrak{q}}}

\newcommand{\Abar}{{\bar A}}
\newcommand{\Ebar}{{\bar E}}
\newcommand{\kbar}{{\bar k}}
\newcommand{\Kbar}{{\bar K}}
\newcommand{\Pbar}{{\bar P}}
\newcommand{\Sbar}{{\bar S}}
\newcommand{\Tbar}{{\bar T}}
\newcommand{\gbar}{{\bar\gamma}}
\newcommand{\lbar}{{\bar\lambda}}
\newcommand{\ybar}{{\bar y}}
\newcommand{\phibar}{{\bar\f}}

\newcommand{\Acal}{{\mathcal A}}
\newcommand{\Bcal}{{\mathcal B}}
\newcommand{\Ccal}{{\mathcal C}}
\newcommand{\Dcal}{{\mathcal D}}
\newcommand{\Ecal}{{\mathcal E}}
\newcommand{\Fcal}{{\mathcal F}}
\newcommand{\Gcal}{{\mathcal G}}
\newcommand{\Hcal}{{\mathcal H}}
\newcommand{\Ical}{{\mathcal I}}
\newcommand{\Jcal}{{\mathcal J}}
\newcommand{\Kcal}{{\mathcal K}}
\newcommand{\Lcal}{{\mathcal L}}
\newcommand{\Mcal}{{\mathcal M}}
\newcommand{\Ncal}{{\mathcal N}}
\newcommand{\Ocal}{{\mathcal O}}
\newcommand{\Pcal}{{\mathcal P}}
\newcommand{\Qcal}{{\mathcal Q}}
\newcommand{\Rcal}{{\mathcal R}}
\newcommand{\Scal}{{\mathcal S}}
\newcommand{\Tcal}{{\mathcal T}}
\newcommand{\Ucal}{{\mathcal U}}
\newcommand{\Vcal}{{\mathcal V}}
\newcommand{\Wcal}{{\mathcal W}}
\newcommand{\Xcal}{{\mathcal X}}
\newcommand{\Ycal}{{\mathcal Y}}
\newcommand{\Zcal}{{\mathcal Z}}

\renewcommand{\AA}{\mathbb{A}}
\newcommand{\BB}{\mathbb{B}}
\newcommand{\CC}{\mathbb{C}}
\newcommand{\FF}{\mathbb{F}}
\newcommand{\GG}{\mathbb{G}}
\newcommand{\NN}{\mathbb{N}}
\newcommand{\PP}{\mathbb{P}}
\newcommand{\QQ}{\mathbb{Q}}
\newcommand{\RR}{\mathbb{R}}
\newcommand{\ZZ}{\mathbb{Z}}

\newcommand{\bfa}{{\mathbf a}}
\newcommand{\bfb}{{\mathbf b}}
\newcommand{\bfc}{{\mathbf c}}
\newcommand{\bfd}{{\mathbf d}}
\newcommand{\bfe}{{\mathbf e}}
\newcommand{\bff}{{\mathbf f}}
\newcommand{\bfg}{{\mathbf g}}
\newcommand{\bfp}{{\mathbf p}}
\newcommand{\bfr}{{\mathbf r}}
\newcommand{\bfs}{{\mathbf s}}
\newcommand{\bft}{{\mathbf t}}
\newcommand{\bfu}{{\mathbf u}}
\newcommand{\bfv}{{\mathbf v}}
\newcommand{\bfw}{{\mathbf w}}
\newcommand{\bfx}{{\mathbf x}}
\newcommand{\bfy}{{\mathbf y}}
\newcommand{\bfz}{{\mathbf z}}
\newcommand{\bfA}{{\mathbf A}}
\newcommand{\bfF}{{\mathbf F}}
\newcommand{\bfB}{{\mathbf B}}
\newcommand{\bfD}{{\mathbf D}}
\newcommand{\bfG}{{\mathbf G}}
\newcommand{\bfI}{{\mathbf I}}
\newcommand{\bfM}{{\mathbf M}}
\newcommand{\bfzero}{{\boldsymbol{0}}}

\newcommand{\Aut}{\operatorname{Aut}}
\newcommand{\Birat}{\operatorname{Birat}}
\newcommand{\codim}{\operatorname{codim}}
\newcommand{\Crit}{\operatorname{Crit}}
\newcommand{\diag}{\operatorname{diag}}
\newcommand{\Disc}{\operatorname{Disc}}
\newcommand{\Div}{\operatorname{Div}}
\newcommand{\Dom}{\operatorname{Dom}}
\newcommand{\End}{\operatorname{End}}
\newcommand{\Fbar}{{\bar{F}}}
\newcommand{\Fix}{\operatorname{Fix}}
\newcommand{\Gal}{\operatorname{Gal}}
\newcommand{\GL}{\operatorname{GL}}
\newcommand{\Hom}{\operatorname{Hom}}
\newcommand{\Index}{\operatorname{Index}}
\newcommand{\Image}{\operatorname{Image}}
\newcommand{\hhat}{{\hat h}}
\newcommand{\Ker}{{\operatorname{ker}}}
\newcommand{\Lift}{\operatorname{Lift}}
\newcommand{\limstar}{\lim\nolimits^*}
\newcommand{\limstarn}{\lim_{\hidewidth n\to\infty\hidewidth}{\!}^*{\,}}
\newcommand{\Mat}{\operatorname{Mat}}
\newcommand{\maxplus}{\operatornamewithlimits{\textup{max}^{\scriptscriptstyle+}}}
\newcommand{\MOD}[1]{~(\textup{mod}~#1)}
\newcommand{\Mor}{\operatorname{Mor}}
\newcommand{\Moduli}{\mathcal{M}}
\newcommand{\Norm}{{\operatorname{\mathsf{N}}}}
\newcommand{\notdivide}{\nmid}
\newcommand{\normalsubgroup}{\triangleleft}
\newcommand{\NS}{\operatorname{NS}}
\newcommand{\onto}{\twoheadrightarrow}
\newcommand{\ord}{\operatorname{ord}}
\newcommand{\Orbit}{\mathcal{O}}
\newcommand{\Per}{\operatorname{Per}}
\newcommand{\Perp}{\operatorname{Perp}}
\newcommand{\PrePer}{\operatorname{PrePer}}
\newcommand{\PGL}{\operatorname{PGL}}
\newcommand{\Pic}{\operatorname{Pic}}
\newcommand{\Prob}{\operatorname{Prob}}
\newcommand{\Qbar}{{\bar{\QQ}}}
\newcommand{\rank}{\operatorname{rank}}
\newcommand{\Rat}{\operatorname{Rat}}
\newcommand{\Resultant}{\operatorname{Res}}
\renewcommand{\setminus}{\smallsetminus}
\newcommand{\sgn}{\operatorname{sgn}} 
\newcommand{\SL}{\operatorname{SL}}
\newcommand{\Span}{\operatorname{Span}}
\newcommand{\Spec}{\operatorname{Spec}}
\newcommand{\Support}{\operatorname{Supp}}
\newcommand{\tors}{{\textup{tors}}}
\newcommand{\Trace}{\operatorname{Trace}}
\newcommand{\trianglebin}{\mathbin{\triangle}} 
\newcommand{\tr}{{\textup{tr}}} 
\newcommand{\UHP}{{\mathfrak{h}}}    
\newcommand{\<}{\langle}
\renewcommand{\>}{\rangle}

\newcommand{\ds}{\displaystyle}
\newcommand{\longhookrightarrow}{\lhook\joinrel\longrightarrow}
\newcommand{\longonto}{\relbar\joinrel\twoheadrightarrow}


\begin{abstract}
The \emph{dynamical degree} of a dominant rational map
$f:\mathbb{P}^N\dashrightarrow\mathbb{P}^N$ is the quantity
$\delta(f):=\lim(\deg f^n)^{1/n}$. We study the variation of dynamical
degrees in 1-parameter families of maps~$f_T$. We make a conjecture
and ask two questions concerning, respectively, the set of~$t$ such
that: (1) $\delta(f_t)\le\delta(f_T)-\epsilon$; (2)
$\delta(f_t)<\delta(f_T)$; (3) $\delta(f_t)<\delta(f_T)$ and
$\delta(g_t)<\delta(g_T)$ for ``independent'' families of maps. We
give a sufficient condition for our conjecture to hold and prove that
the condition is true for monomial maps. We describe non-trivial
families of maps for which our questions have affirmative and negative
answers.
\end{abstract}


\maketitle

\begin{center}
\large\emph{In honor of Robert Tijdeman's 75th year}
\end{center}


\section{Introduction}

Let $f:\PP^N\dashrightarrow\PP^N$ be a dominant rational map.
A fundamental invariant attached to~$f$ is its (\emph{first})
\emph{dynamical degree}, which is the quantity
\[
  \d(f) = \lim_{n\to\infty} \Bigl(\deg(f^n)\Bigr)^{1/n}.
\]
We note that the convergence of the limit is an easy convexity
argument using the fact that~$\deg(f^{n+m})\le\deg(f^n)\deg(f^m)$,
see for example~\cite[Proposition~9.6.4]{MR1326374}, and
we recall that~$f$ is said to be \emph{algebraically stable} if
$\d(f)=\deg(f)$, which in turn is equivalent to
$\deg(f^n)=\d(f)^n=(\deg f)^n$ for all $n\ge1$.

In this paper we study the variation of dynamical degrees as~$f$
moves in a family. We consider a smooth irreducible quasi-projective
curve~$T/\CC$ and a family
\[
  f_T:\PP^N_T\dashrightarrow\PP^N_T
\]
of dominant rational maps, i.e., for every~$t\in T(\CC)$, the
specialization~$f_t$ is a dominant rational map.  We start with a
conjecture and two questions, followed by some brief remarks.  Our
main results include a proof of the conjecture for monomial maps and
the analysis of several non-trivial families of maps which display
some of the subtleties inherent in our two questions.

\begin{conjecture}
\label{conjecture:1}
For all $\e>0$, the set
\[
  \bigl\{ t\in T(\CC) : \d(f_t) \le \d(f_T)-\e \bigr\}
\]
is finite.
\end{conjecture}

\begin{question}  
\label{question:2}
Suppose that~$T$ and~$f_T$ are defined over~$\Qbar$. Under what
circumstances is the \emph{exceptional set}
\[
  \Ecal(f_T) :=
  \bigl\{ t\in T(\Qbar) : \d(f_t) < \d(f_T) \bigr\}
\]
a set of bounded height?
\end{question}

\begin{question}
\label{question:3}
Let $g_T:\PP^N_T\dashrightarrow\PP^N_T$ be another family of dominant
rational maps, and let~$\Ecal(f_T)$ and~$\Ecal(g_T)$ be exceptional
sets as defined in Question~$\ref{question:2}$. Under what
circumstances does the following implication
hold\textup{:}\footnote{We recall that the \emph{symmetric set
    difference} of two sets~$A$ and~$B$ is the set $A\trianglebin B :=
  (A\cup B)\setminus (A\cap B)$, or alternatively $A\trianglebin B :=
  (A\setminus B)\cup(B\setminus A)$.}
\[
  \text{$\Ecal(f_T)\cap\Ecal(g_T)$ is infinite}
  \quad\Longrightarrow\quad
  \text{$\Ecal(f_T)\trianglebin \Ecal(g_T)$ is finite?}
\]
\end{question}

Conjecture~\ref{conjecture:1} is inspired by
Xie~\cite[Theorem~4.1]{arxiv1106.1825}, a special case of which
implies that Conjecture~\ref{conjecture:1} is true for families of
birational maps of~$\PP^2$.\footnote{In a private communication, Xie
  has indicated that the methods used in~\cite{arxiv1106.1825} can
  be used prove Conjecture~\ref{conjecture:1} for dominant
  rational self-maps of~$\PP^2$.} Our primary goal in this paper is to
provide justification for studying Questions~\ref{question:2}
and~\ref{question:3} by analyzing in depth an interesting
three-parameter family of rational maps and showing that the
questions are true for one-parameter subfamilies. The maps
$f_{a,b,c}:\PP^2\dashrightarrow\PP^2$ that we study are defined by
\begin{equation}
  \label{eqn:fabcintro}
  f_{a,b,c}\bigl([X,Y,Z]\bigr)=[XY,XY+aZ^2,bYZ+cZ^2].
\end{equation}
For $abc\ne0$, we first show that $\d(f_{a,b,c})<2$ if and only if
there is a root of unity~$\xi$ with the property that
$c^2=(\xi+\xi^{-1})^2ab$; cf.\ Theorem~\ref{theorem:ZKabc}.
Taking~$a,b,c$ to be polynomials in one variable, we use this
criterion to prove Questions~\ref{question:2}
and~\ref{question:3} for 1-parameter subfamilies of the
family~\eqref{eqn:fabcintro}.

\begin{theorem}
Let~$f_{a,b,c}:\PP^2\dashrightarrow\PP^2$ be the map~\eqref{eqn:fabcintro}.
\begin{parts}
\Part{(a)}
\textup{(Corollary \ref{corollary:fatbtct}):}\enspace
Let $a(T),b(T),c(T)\in\Qbar[T]$ be non-zero polynomials
satisfying $\d(f_{a(T),b(T),c(T)})=2$. Then the exceptional set
\[
  \Ecal(f_{a(T),b(T),c(T)})
  = \bigl\{ t\in\Qbar : \d(f_{a(t),b(t),c(t)}) < 2 \bigr\}
\]
is a set of bounded height.
\Part{(b)}
\textup{(Theorem \ref{theorem:unlikelyfabc}):}\enspace
Let $a_1(T),b_1(T),c_1(T),a_2(T),b_2(T),c_2(T)\in\Qbar[T]$ be non-zero
polynomials such that 
\[
  \text{$\d(f_{a_1(T),b_1(T),c_1(T)})=2$\quad and\quad
    $\d(f_{a_2(T),b_2(T),c_2(T)})=2$.}
\]
Then
\begin{multline*}
  \#\Bigl(\Ecal(f_{a_1(T),b_1(T),c_1(T)})\cap\Ecal(f_{a_2(T),b_2(T),c_2(T)})\Bigr)
     =\infty\\
  \quad\Longrightarrow\quad
  \#\Bigl(\Ecal(f_{a_1(T),b_1(T),c_1(T)})\trianglebin
  \Ecal(f_{a_2(T),b_2(T),c_2(T)})\Bigr) < \infty.
\end{multline*}
\end{parts}
\end{theorem}

\begin{remark}
We observe that Conjecture~\ref{conjecture:1} and Question~\ref{question:3}
appear to be geometric, since they are stated over~$\CC$, while
Question~\ref{question:2} is clearly arithmetic in nature. This
dichotomy is, however, somewhat misleading, since proofs of unlikely
intersection statements such as Question~\ref{question:3}
invariably require a considerable amount of arithmetic. On the other
hand, Conjecture~\ref{conjecture:1} may well admit a geometric proof.
\end{remark}

\begin{remark}
We note that Question~\ref{question:3} should be only half the
story.  The other half would be a statement saying that if
$\Ecal(f_T)\cap\Ecal(g_T)$ is infinite, then~$f_T$ and~$g_T$ are
``geometrically dependent.'' We do not currently know how to formulate
this precisely.
\end{remark}

\begin{remark}
\label{remark:xieexample}
The conjectures, questions, examples, and results in this paper were inspired by
work of Xie~\cite{arxiv1106.1825}. In particular, he proves a
beautiful theorem on the reduction modulo~$p$ of a birational
map~$f:\PP^2_\QQ\dashrightarrow\PP^2_\QQ$. In the context of
``degeneration of dynamical degree in families,'' Xie's map~$f$ should
be viewed as a family of maps over~$T=\Spec\ZZ$, and the reduction
$\tilde f_p:\PP^2_{\FF_p}\dashrightarrow\PP^2_{\FF_p}$ of~$f$
modulo~$p$ is the specialization of~$f$ to the fiber over~$p$.
Xie~\cite{arxiv1106.1825} proves that
\[
  \lim_{p\to\infty} \d(\tilde f_p) = \d(f).
\]
One might suspect that in fact $\d(\tilde f_p) = \d(f)$ for all
sufficiently large primes~$p$, but Xie gives an intriguing
example~\cite[Section~5]{arxiv1106.1825} of a birational map
\[
  f:\PP^2_\QQ\longrightarrow\PP^2_\QQ,\quad
  f\bigl([X,Y,Z]\bigr)=[XY,XY-2Z^2,YZ+3Z^2]
\]
having the property that there is a strict inequality $\d(\tilde
f_p)<\d(f)$ for all primes~$p$.
\end{remark}

A fundamental inequality from~\cite{arxiv1106.1825} that Xie uses to
study dynamical degrees in families says that there is an absolute
constant~$\g>0$ such that
\[
  \d(f) \ge \g\cdot \frac{\deg(f^2)}{\deg(f)}
  \quad\text{for all birational maps $f:\PP^2\dashrightarrow\PP^2$.}
\]
The crucial point here is that~$\g$ is independent of~$f$, so for
example, one can replace~$f$ by~$f^n$ without changing~$\g$. We ask
whether such estimates hold more generally.

\begin{conjecture}
\label{conjecture:4}
Let $N\ge1$. There exists a constant~$\g_N>0$ such that
for all dominant rational maps $f:\PP^N\dashrightarrow\PP^N$ we have
\[
  \d(f) \ge \g_N\cdot \min_{0\le k< N}  \frac{\deg(f^{k+1})}{\deg(f^{k})}. 
\]
\end{conjecture}

It is possible that Conjecture~\ref{conjecture:4} is too optimistic,
and we should instead take a minimum over $0\le k\le\kappa(N)$ for
some upper index that grows more rapidly with~$N$, but we will at
least prove that Conjecture~\ref{conjecture:4} is true as stated for
monomial maps. More precisely, we prove that if
$f:\PP^N\dashrightarrow\PP^N$ is a dominant monomial map, then
Conjecture~\ref{conjecture:4} holds with $\g_N=(2^{1/N}-1)/2N^2$; see
Section~\ref{section:monomialmaps}.


We briefly summarize the contents of this article:
\begin{parts}
\Part{\S\ref{section:anexample}}
We study the geometry and algebraic stability
of the family of maps~$f_{a,b,c}$
defined by~\eqref{eqn:fabcintro},
and we answer Question~\ref{question:2} affirmatively for these families when~$a,b,c$
are polynomials of one variable.
\Part{\S\ref{section:unlikelyint}}
We answer Question~\ref{question:3} affirmatively for a pair of
families of maps $f_{a_1,b_1,c_1}$ and~$f_{a_2,b_2,c_2}$, where the~$a_i,b_i,c_i$
are again polynomials of one variable.
\Part{\S\ref{section:neganswers}}
We describe families of maps shown to us by Junyi Xie for which
Questions~\ref{question:2} and~\ref{question:3} have negative answers. 
\Part{\S\ref{section:conj4implyconj1}}
We sketch a proof (essentially due to Xie) that
Conjecture~\ref{conjecture:4} implies Conjecture~\ref{conjecture:1},
more generally over higher dimensional base varieties; see
Theorem~\ref{theorem:thm4implythm1}.
\Part{\S\ref{section:monomialmaps}}
We prove Conjecture~\ref{conjecture:4} for monomial maps.
\end{parts}

\begin{acknowledgement}
The authors would like to thank Serge Cantat for providing an example
showing that our original version of Conjecture~\ref{conjecture:4} was
too optimistic and that our original proof of
Conjecture~\ref{conjecture:4} for monomial maps was incorrect, Mattias
Jonsson for showing us the proof of
Proposition~\ref{proposition:degfinvledegfn-1}, and Wonwoong Lee for
pointing out an error in the proof of
Proposition~\ref{proposition:invim}(c).  And most importantly, we
thank Junyi Xie for providing the examples described in
Section~\ref{section:neganswers} that dashed our original expectation
that Questions~\ref{question:2} and~\ref{question:3} would always have
affirmative answers, as well as Xie's simplification of our original
proof of Theorem~\ref{theorem:unlikelyfabc}.
\end{acknowledgement}

\section{A Bounded Height Example}
\label{section:anexample}

In this section we study a family of rational maps inspired by Xie's
map~\cite[Section~5]{arxiv1106.1825} described in
Remark~\ref{remark:xieexample}.  We set the following notation.

\begin{definition}
Let $\Rcal$ be an integral domain with field of fractions~$\Kcal$. For
each triple $a,b,c\in\Rcal$, let $f_{a,b,c}:\PP^2_\Rcal\to\PP^2_\Rcal$
be the rational map
\[
  f_{a,b,c}\bigl([X,Y,Z]\bigr)=[XY,XY+aZ^2,bYZ+cZ^2].
\]
We also define the set of \emph{exceptional triples} to be
\[
  \Zcal(\bar\Kcal) = 
     \left\{ (a,b,c)\in\AA^3(\bar\Kcal) : 
           \begin{array}{c}
               \z c^2 + (\z+1)^2ab = 0~\text{for some}\\[.5\jot]
               \text{root of unity $\z\in\bar\Kcal$}\\
           \end{array} \right\}.
\]
(We note that replacing~$\z$ by~$\z^2$, we could alternatively
define~$\Zcal(\bar\Kcal)$ to be the set of triples satisfying
$c^2=-(\z+\z^{-1})^2ab$.)
\end{definition}

\begin{theorem}
\label{theorem:ZKabc}
Let $(a,b,c)\in\AA^3(\bar\Kcal)$ with $abc\ne0$. Then
\[
  \text{$f_{a,b,c}$ is algebraically stable}
  \quad\Longleftrightarrow\quad
  (a,b,c)\notin\Zcal(\bar\Kcal).
\]
\end{theorem}

\begin{corollary}
\label{corollary:fatbtct}
Let $a(T),b(T),c(T)\in\Qbar[T]$ be non-zero polynomials such that
$f_{a(T),b(T),c(T)}$ is algebraically stable. Then
\[
  \bigl\{ t\in\Qbar : 
   \text{$f_{a(t),b(t),c(t)}$ is not algebraically stable} \bigr\}
\]
is a set of bounded height.
\end{corollary}

The key to proving Theorem~\ref{theorem:ZKabc} is an analysis of the
geometry of the map~$f_{a,b,c}$.

\begin{proposition}
\label{proposition:invim}
Let $a,b,c\in\Kcal$ with $abc\ne0$.
\begin{parts}
\Part{(a)}
The map $f_{a,b,c}$ is birational, and its indeterminacy locus
is the set
\[
  I(f_{a,b,c}) = \bigl\{ [0,1,0], [1,0,0] \bigr\}.
\]
\Part{(b)}
The critical locus of~$f$ is the set
\[
  \Crit(f) = \{Y=0\} \cup \{Z=0\}.
\]
\Part{(c)}
Let $[\a,\b,\g]\in\PP^2(\bar\Kcal)$. Then the set
$f_{a,b,c}^{-1}\bigl([\a,\b,\g]\bigr)$ consists of a single point
except in the following situations\textup:
\begin{align*} 
  f_{a,b,c}^{-1}\bigl([0,a,c]\bigr) &= \{Y=0\}\setminus\bigl\{[1,0,0]\bigr\},\\
  f_{a,b,c}^{-1}\bigl([1,1,0]\bigr) & =\{Z=0\}\setminus\bigl\{ [0,1,0], [1,0,0] \bigr\},\\
  f_{a,b,c}^{-1}\bigl([0,0,1]\bigr) & =\emptyset,\\
  f_{a,b,c}^{-1}\bigl([a,at,ct-c]\bigr) & =\emptyset~\text{for all $t$.}
\end{align*}
\end{parts}
\end{proposition}
\begin{proof}
To ease notation, we let $f=f_{a,b,c}$.
\par\noindent(a)\enspace
Let
\[
  g(X,Y,Z) = \bigl(ab^2X(Y-X),(cX-cY+aZ)^2,b(cX-cY+aZ)(Y-X)\bigr).
\]
Then an easy calculation in affine coordinates gives
\[
   g\circ f(X,Y,Z) = a^2b^2YZ^2\cdot (X,Y,Z), 
\]
which shows that~$f$ is birational with~$f^{-1}$ induced by~$g$;
cf.\ \cite[Section~5]{arxiv1106.1825}.\footnote{As
  in~\cite{arxiv1106.1825}, we can decompose~$f=f_2\circ f_1$ with
  biratonal maps $f_1= [XY,XY+bYZ,Z^2]$ and $f_2=[X,X+aZ,-X+Y+cZ]$.}
The indeterminacy locus of~$f$ is the set where
\[
  XY = XY+aZ^2 = bYZ+cZ^2 = 0.
\]
We note that $XY=0$ forces $aZ^2=0$, and hence $Z=0$ under our
assumption that~$a\ne0$. This gives two possible points in~$I(f)$,
namely~$[0,1,0]$ and~$[1,0,0]$, and it is clear that these  points
are in~$I(f)$.
\par\noindent(b)\enspace
The critical locus~$\Crit(f)$ of~$f$ is the set where
\[
  \det\begin{pmatrix}Y & X & 0 \\ Y & X & 2aZ \\ 0 & bZ & bY+2cZ^2\\ \end{pmatrix}
  = -2abYZ^2=0.
\]  
\par\noindent(c)\enspace
It is a standard fact that if $\G\subset\PP^2$ is a curve with the
property that $f(\G)$ is a point, then necessarily
$\G\subseteq\Crit(f)$; see for
example~\cite[Lemma~23(c)]{arxiv1607.05772}.  We compute
\[
  \{Y=0\} \xrightarrow{\;f\;} [0,a,c]
  \quad\text{and}\quad
  \{Z=0\} \xrightarrow{\;f\;} [1,1,0] \in \Fix(f_{a,b,c}).
\]
An easy calculation shows that
\[
  f^{-1}\bigl([0,a,c]\bigr) \subset \{Y=0\}
  \quad\text{and}\quad
  f^{-1}\bigl([1,1,0]\bigr) \subset \{Z=0\},
\]
so the inverse images are the indicated sets with~$I(f)$ removed.

It remains to determine for which points~$P=[\a,\b,\g]$ the inverse
image~$f^{-1}(P)$ is empty.  We include the proof, although we do
require this result in the sequel.

First, we have
\begin{align*}
f^{-1}\bigl([0,0,1]\bigr)
&\subseteq \bigl\{ [X,Y,Z] : XY=XY+aZ^2=0 \bigr\}\\
&=\bigl\{ [1,0,0],[0,1,0] \bigr\} = I(f),
\end{align*}
which shows that $f^{-1}\bigl([0,0,1]\bigr)=\emptyset$.

Next we consider $\a=0$ and $\b\ne0$ and compute
\[
  f^{-1}\bigl([0,\b,\g]\bigr)
  \subseteq \bigl\{ [X,Y,Z] : XY=0 \bigr\}
  = \bigl\{ [0,Y,Z] \bigr\} \cup \bigl\{ [X,0,Z] \bigr\}.
\]
Since
\[
  f\bigl([X,0,Z]\bigr) = [0,aZ^2,cZ^2] =
  \begin{cases}
    [0,a,c] &\text{if $Z\ne0$,} \\
    [0,0,1] \in I(f)&\text{if $Z=0$,}\\
    \end{cases}
\]
we have already dealt with this case. The other case
yields (note that $Z\ne0$, since we're assuming that $a\ne0$ and $\b\ne0$)
\[
  [0,\b,\g] = f\bigl([0,Y,Z]\bigr) = [0,aZ^2,bYZ+cZ^2] =[ 0,aZ,bY+cZ],
\]
so  solving for~$[Y,Z]$ gives
\[
  f^{-1}\bigl([0,\b,\g]\bigr) = [0,a\g-c\b,b\b],
\]
and thus $f^{-1}\bigl([0,\b,\g]\bigr)$ consists of a single point
if~$\b\ne0$.

It remains to consider $f^{-1}\bigl([\a,\b,\g]\bigr)$ with $\a\ne0$.
This set is contained in the set of $[X,Y,Z]$ satisfying the
simultaneous equations
\begin{equation}
  \label{eqn:axybxy}
  \a (X Y + a Z^2)   = \b X Y \quad\text{and}\quad
  \a (b Y Z + c Z^2) = \g X Y.
\end{equation}
We're assuming that $\a\ne0$, and the first coordinate of~$f$ is~$XY$,
so $XY\ne0$. We also note (using $a\a\ne0$ and $XY\ne0$)
\begin{align*}
  \a=\b
  \quad\Longrightarrow\quad
  Z=0
  \quad\Longrightarrow\quad
  \g XY = 0
  &\quad\Longrightarrow\quad 
  \g = 0\\
  &\quad\Longrightarrow\quad
  [\a,\b,\g]=[1,1,0],
\end{align*}
a case with which we have already dealt. So we may assume that
$\a\ne\b$, as well as~$\a\ne0$.

We eliminate~$X$ from the two equations in~\eqref{eqn:axybxy}. Thus
\begin{multline*}
0 = \g  \Bigl(  \overbrace{\a (X Y + a Z^2)   - \b X Y}^{\text{vanishes from \eqref{eqn:axybxy}}} \Bigr)
+
(\a - \b ) \Bigl(  \overbrace{\a (b Y Z + c Z^2) - \g X Y}^{\text{vanishes from \eqref{eqn:axybxy}}} \Bigr) \\
 = \a Z  \Bigl( (\a c   - \b c + \g a)Z + (\a-\b) b Y \Bigr).
\end{multline*}
We can rule out $Z=0$, since we're assuming that $\a\ne\b$, and
$f\bigl([X,Y,0]\bigr)=[XY,XY,0]\ne[\a,\b,0]$. Since we are in a case
with $(\a-\b)bY\ne0$, we see that $\a c - \b c + \g a=0$ leads to a
contradiction, so
\[
  \a c - \b c + \g a = 0 \quad\Longrightarrow\quad
  f^{-1}\bigl([\a,\b,\g]\bigr)=\emptyset.
\]
The relation $\a c - \b c + \g a=0$ with $\a\ne0$ is equivalent to the
point $[\a,\b,\g]$ having the form $[a,at,ct-c]$ for some~$t$, which
proves half of what we want. On the other hand, if $\a c - \b c + \g a\ne0$,
and continuing with the assumption $\a(\a-\b)abc\ne0$, we claim that
$f^{-1}\bigl([\a,\b,\g]\bigr)$ contains exactly one point. Indeed, we
can solve uniquely for
\[
  [Y,Z] = \bigl[\a c- \b c + \g a, (\a-\b)b \bigr],
\]
andt then~\eqref{eqn:axybxy} determines~$X$, which yields
\[
f^{-1}\bigl([\a,\b,\g]\bigr) =
[ -\a (\a-\b) a b^2,
  (\a c - \b c + \g a)^2,
 -(\a - \b) (\a c - \b c + \g a)b ].
\]
This completes the proof of Proposition~\ref{proposition:invim}(c).
\end{proof}

\begin{proof}[Proof of Theorem \textup{\ref{theorem:ZKabc}}]
To ease notation, we let $f=f_{a,b,c}$. The map~$f$ is not
algebracally stable if and only if there is a curve~$\G\subset\PP^2$
and an $N\ge1$ such that $f^N(\G)\subset I(f)$. 

In particular, if~$f$ is not algebracally stable, then there is some $0\le
n\le N-1$ such that $\dim f^n(\G)=1$ and $\dim
f^{n+1}(\G)=0$. Proposition~\ref{proposition:invim} tells us that the
only curves that~$f$ collapses are the curves~$Y=0$
and~$Z=0$. Further,
\[
  f\bigl(\{Z=0\}\bigr)=[1,1,0] \in \Fix(f),
\]
so if $f^n(\G)=\{Z=0\}$, then for all $k\ge1$ we have
$f^{n+k}(\G)=[1,1,0]\notin I(f)$. So we have shown that~$f$ is not
algebraically stable if and only if there is a curve~$\G$ and
integers~$N>n\ge0$ such that
\begin{equation}
  \label{eqn:fnGY0fNG010}
  f^n(\G)=\{Y=0\} 
 \quad\text{and}\quad
  f^N(\G)=[0,1,0].
\end{equation}
(We can't have $f^N(\G)$ equal to the other point $[1,0,0]$ in $I(f)$,
since $f\bigl(\{Y=0\}\bigr)=[0,a,c]$ and
$f\bigl(\{X=0\}\bigr)=\{X=0\}$, so once we get to a point with $X=0$,
applying~$f$ never gets back to a point with $X\ne0$.)

We observe that~\eqref{eqn:fnGY0fNG010} is true if and only if
\begin{align*}
  [0,1,0] 
  =  f^N(\G) 
  &= (f^{N-n-1}\circ f \circ f^n)(\G)\\
  &= (f^{N-n-1}\circ f) \bigl(\{Y=0\}\bigr)\\
  &= f^{N-n-1}\bigl([0,a,c]\bigr).
\end{align*}
So we have proven that $f$ is not algebraically stable if and only if
there is an $n\ge0$ such that the $Z$-coordinate of
$f^n\bigl([0,a,c]\bigr)$ vanishes. We also note that
$f\bigl([0,0,1]\bigr)=[0,a,c]$, so we may as well start
at~$[0,0,1]$. This prompts us to define polynomials
$U_n,V_n\in\ZZ[a,b,c]$ by the formula
\[
  [0,U_n,V_n] = f^n\bigl([0,0,1]\bigr).
\]
Then we have shown that
\[
  \text{$f_{a,b,c}$ is not algebraically stable}
  \;\Longleftrightarrow\;
  \text{$V_n(a,b,c)=0$ for some $n\ge1$.}
\]

We now observe that~$(V_n)_{n\ge0}$ is a linear recurrence, at least
until reaching a term that vanishes. Indeed, we have
\begin{align*}
  [0,U_{n+1},V_{n+1}]
  = f\bigl([0,U_n,V_n]\bigr)
  &= [0,aV_n^2,bU_nV_n+cV_n^2]\\
  &= [0,aV_n,bU_n+cV_n].
\end{align*}
Thus
\[
  \begin{pmatrix} U_{n+1} \\ V_{n+1}  \\ \end{pmatrix}
  =
  \begin{pmatrix} 0 & a \\ b & c \\ \end{pmatrix}
  \begin{pmatrix} U_n \\ V_n  \\ \end{pmatrix},
\]
and repeated application together with the initial value
$(U_0,V_0)=(0,1)$ gives the matrix formula
\[
  \begin{pmatrix} U_n \\ V_n  \\ \end{pmatrix}
  =
  \begin{pmatrix} 0 & a \\ b & c \\ \end{pmatrix}^n
  \begin{pmatrix} 0 \\ 1  \\ \end{pmatrix},
\]

Letting~$\l$ and~$\lbar$ be the eigenvalues
of~$\left(\begin{smallmatrix}0&a\\b&c\\\end{smallmatrix}\right)$,
  i.e., the roots of
\[
  T^2 - cT - ab = 0,
\]
an elementary linear algebra calculation yields
\[
  V_n = \frac{\l^{n+1}-\lbar^{n+1}}{\l-\lbar}.
\]
(Unless $c^2+4ab=0$, which we will deal with later.) So~$f$ is not
algebraically stable if and only if there is some $n\ge1$ such that
$\l^n=\lbar^n$. Writing~$\l$ and $\lbar$ explicitly, we find that~$f$
is not algebraically stable if and only if~$(a,b,c)$ satisfies
\[
  c+\sqrt{c^2+4ab} = \z \bigl(c-\sqrt{c^2+4ab}\bigr)
  \quad\text{for some root of unity $\z\in\Qbar$.}
\]
A little algebra yields
\[
  \z c^2 + (\z+1)^2 ab = 0,
\]
which is the desired result.
  
It remains to deal with the case that $c^2+4ab=0$, i.e., $\l=\lbar$.
But then an easy calculation shows that
\[
  V_n = (n+1)(c/2)^n,
\]
so $V_n$ never vanishes under our assumption that $abc\ne0$.
\end{proof}

\begin{proof}[Proof of Corollary $\ref{corollary:fatbtct}$]
According to Theorem~\ref{theorem:ZKabc}, the map $f_{a(t),b(t),c(t)}$ is
not algebraically stable if and only if there is a root of unity
$\z\in\Qbar$ with the property that
\[
  \z c(t)^2 + (\z+1)^2a(t)b(t) = 0.
\]
As $\z$ varies over roots of unity, the polynomials 
\begin{equation}
  \label{eqn:ct2xx2}
  \z c(T)^2 + (\z+1)^2a(T)b(T)\in\Qbar[T]
\end{equation}
have bounded degree (depending on~$a,b,c$) and have coefficients of
bounded height.  But the heights of the roots of a polynomial are easily
bounded in terms of the degree and the heights of the coefficients;
see for example \cite[Theorem~VIII.5.9]{MR2514094}.  Hence the roots of the
polynomials~\eqref{eqn:ct2xx2} have height bounded independently
of~$\z$.
\end{proof}


\section{An Unlikely Intersection Example}
\label{section:unlikelyint}
Our goal in this section is to give an affirmative answer to a
non-trivial case of Question~\ref{question:3} for the intersection of
the exceptional sets of two maps.

\begin{definition}
Let $\Rcal$ be an integral domain with field of fractions~$\Kcal$, and
for non-zero $a,b,c\in\Rcal$, let
$f_{a,b,c}:\PP^2_\Rcal\to\PP^2_\Rcal$ denote the map
\[
  f_{a,b,c}\bigl([X,Y,Z]\bigr)=[XY,XY+aZ^2,bYZ+cZ^2]
\]
that we already studied in Section~\ref{section:anexample}.  We define
the \emph{exceptional set of~$f_{a,b,c}$} to be the set of prime
ideals
\[
  \Ecal(f_{a,b,c}) 
  = \bigl\{ \gp\in\Spec(\Rcal) : 
           \d(\tilde f_{a,b,c}\bmod\gp) < \d(f_{a,b,c}) \bigr\}.
\]
\end{definition}

\begin{theorem}
\label{theorem:unlikelyfabc}
Let $a_1,b_1,c_1,a_2,b_2,c_2\in\Qbar[T]$ be non-zero polynomials.
For $i=1,2$,  let
\[
  f_{a_i,b_i,c_i}\bigl([X,Y,Z]\bigr)=[XY,XY+a_iZ^2,b_iYZ+c_iZ^2]
\]
be the associated families of rational maps, and assume that they are
algebraically stable as maps over the function field~$\Qbar(T)$.  Then
\begin{multline*}
  \#\Bigl( \Ecal(f_{a_1,b_1,c_1})\cap \Ecal(f_{a_2,b_2,c_2}) \Bigr) = \infty \\
  \Longrightarrow\quad
  \#\Bigl( \Ecal(f_{a_1,b_1,c_1})\trianglebin \Ecal(f_{a_2,b_2,c_2}) \Bigr) < \infty.
\end{multline*}
\end{theorem}

\begin{proof}
We recall that Theorem~\ref{theorem:ZKabc} says that
if~$a,b,c\in\Rcal$ are non-zero and if~$f_{a,b,c}$ is algebraically
stable, i.e., $\d(f_{a,b,c})=2$, then
\begin{equation*}
  \Ecal(f_{a,b,c}) = \left\{
       \gp\in\Spec(\Rcal) : 
           \begin{array}{c}
               \z c^2 + (\z+1)^2ab \equiv 0~\bigl(\text{mod}~\gp\Rcal[\z]\bigr)\\[.5\jot]
               \text{for some root of unity $\z\in\bar\Kcal$}\\
           \end{array} \right\}.
\end{equation*}
We apply this with $\Rcal=\Qbar[T]$.

To ease notation, for $i=1$ and~$2$, we let $F_i = f_{a_i,b_i,c_i}$,
and we define maps
\[
  \f_i : \PP^1_\Qbar\longrightarrow\PP^1_\Qbar,\quad
  \f_i(t) = \frac{c_i(t)^2}{a_i(t)b_i(t)}.
\]
Also let $S_i=\{t\in\Qbar:a_i(t)b_i(t)c_i(t)=0\}$, so in particular
each~$S_i$ is a finite set.

Let
\[ 
  \bfmu := \bigl\{ [\z,1] \in \PP^1(\Qbar) : \text{$\z$ is a root of unity} \bigr\},
\]
and consider the rational map
\[
  \psi : \PP^1_\Qbar\longrightarrow\PP^1_\Qbar,\quad
  \psi(z)= -\frac{(z+1)^2}{z}.
\]
Then the fact that $\psi(1/z)=\psi(z)$ implies that the set
$S:=\psi(\bfmu)$ satisfies $\psi^{-1}(S)=\bfmu$.

Theorem~\ref{theorem:ZKabc} tells us that
\[
  \Ecal(F_i)\setminus S_i=\f_i^{-1}(S)\setminus S_i,
\]
so we need to show that
\[
  \#\bigl(\f_1^{-1}(S)\cap\f_2^{-1}(S)\bigr)=\infty
  \quad\Longrightarrow\quad
  \f_1^{-1}(S)=\f_2^{-1}(S).
\]

So we suppose that
$\#\bigl(\f_1^{-1}(S)\cap\f_2^{-1}(S)\bigr)=\infty$.  Define
\[
  \F:=(\f_1,\f_2):\PP^1_\Qbar\longrightarrow\PP^1_\Qbar\times\PP^1_\Qbar,
  \quad\text{and let}\quad
  C = \F(\PP^1).
\]
Then~$C$ is an irreducible curve and $\#\bigl(C\cap(S\times
S)\bigr)=\infty$. Next consider the finite morphism
\[
  G :=(\psi,\psi) : 
  \PP^1_\Qbar\times\PP^1_\Qbar\longrightarrow\PP^1_\Qbar\times\PP^1_\Qbar,
\]
and let~$V$ be any irreducible component of~$G^{-1}(C)$, so in
particular, $G(V)=C$. Then our assumption implies that the set
\[
  V\cap(\bfmu\times\bfmu) = (G|_V)^{-1}\bigl(C\cap(S\times S)\bigr)
\]
is infinite.

A famous result of Ihara, Serre, and Tate \cite[Chapter~8,
  Theorem~6.1]{lang:diophantinegeometry} says that a curve
in~$\GG_m^2(\Qbar)$ contains infinitely many torsion points if and
only if the curve contains a torsion-point-translate of a subtorus
of~$\GG_m^2$.  Hence there exists a pair of integers~$(m,n)$ satisfying
$\gcd(m,n)=1$ and a root of unity~$\z$ such that
\[
  V = \bigl\{ [x,y,1] : x^my^n = \z\bigr\}.
\]
It follows that
\[
  V\cap (\bfmu\times\bfmu) = V\cap (\bfmu\times\PP^1) = V\cap (\PP^1\times\bfmu),
\]
from which we conclude that
\begin{align*}
  C\cap(S\times S)
  &=G\bigl(V\cap G^{-1}(S\times S)\bigr)\\
  &=G\bigl(V\cap(\bfmu\times\bfmu)\bigr)
  =G\bigl(V\cap(\bfmu\times\PP^1)\bigr)
  = C\cap (S\times\PP^1).
\end{align*}
A similar calculation gives $C\cap(S\times S)=C\cap (\PP^1\times S)$.
Hence
\[
\f^{-1}_1(S)
= \F^{-1}\bigl(C\cap(S\times\PP^1)\bigr)
= \F^{-1}\bigl(C\cap(\PP^1\times S)\bigr)
= \f^{-1}_2(S),
\]
which concludes the proof of Theorem~\ref{theorem:unlikelyfabc}.
\end{proof}

\section{Families that Give Negative Answers to Questions \ref{question:2} and \ref{question:3}}
\label{section:neganswers}
In this section we study a family of maps, shown to us by Junyi Xie,
that yield negative answers to Questions~\ref{question:2}
and~\ref{question:3}. For $a\ne0$ and~$b$, we define a family of
rational maps
\[
  g_T = g_{a,b,T} : \PP^2\dashrightarrow\PP^2
\]
by the formula
\begin{equation}
  \label{eqn:gtaxbz}
  g_T\bigl([X,Y,Z]\bigr)
  = \bigl[ (aX+bZ)(X-TZ)+(X-Z)Y,(X-Z)Y,(X-TZ)Z \bigr].
\end{equation}
For generic~$T$, it is easy to check that~$g_T$ is a birational map with
indeterminacy loci
\[
  I(g_T) = \bigl\{ [0,1,0], [T,0,1] \bigr\},\quad
  I(g_T^{-1}) = \bigl\{ [1,1,0], [a+b,0,1] \bigr\}.
\]
Thus for example, we find that
\begin{align*}
g_T^{-1}(1,1,0) &= \{X=TZ\} \setminus [0,1,0],\\
g_T^{-1}(a+b,0,1) &= \{X=Z\} \setminus [0,1,0].
\end{align*}
The lines
\[
  L := \{Y=0\} \quad\text{and}\quad H :=\{Z=0\}
\]
are  $g_T$-invariant, and the action of~$g_T$ on these
lines is given by
\begin{align*}
  g_T|_L(X,0,Z) &= [aX+bZ,0,Z],&&\text{with $g_T$ undefined at $[T,0,1]$,} \\
  g_T|_H(X,Y,0) &= [aX+Y,Y,0],&&\text{with $g_T$ undefined at $[0,1,0]$.}
\end{align*}
From this information it is easy to see that~$g_T$ is generically algebraically
stable, i.e., $\d(g_T)=2$, and that the exceptional set of~$g_T$ is
\begin{align*}
  \Ecal(g_T)
  :&= \bigl\{ t\in\Qbar : \d(g_t) < \d(g_T) \bigr\} \\
  &= \bigl\{ t\in\Qbar : g_t^n(1,0,1)=[t,0,1]~\text{for some $n\ge0$} \bigr\}.
\end{align*}
(Note that $t=1$ yields $\deg(g_1)=1$, so $\d(g_1)<\d(g_T)$.)
We compute
\[
  g_t^n(1,0,1) = \bigl[a^n+b(a^{n-1}+a^{n-2}+\cdots+a+1),0,1\bigr].
\]
Hence
\[
  \Ecal(g_{a,b,T}) = \bigl\{ a^n+b(a^{n-1}+a^{n-2}+\cdots+a+1) : n \ge 0 \bigr\}.
\]

In particular, taking $a=1$ and $b\ne0$ gives
\[
  \Ecal(g_{1,b,T}) = \{1,1+b,1+2b,1+3b,\ldots\}.
\]
This proves the following result.

\begin{proposition}
Let $g_{a,b,T}$ be the family of maps~\eqref{eqn:gtaxbz}
\begin{parts}
\Part{(a)}
The family $g_{1,1,T}$ has exceptional set
$\Ecal(g_{1,1,T})=\{1,2,3,4,\cdots\}$, which is clearly a set of
unbounded height. Hence~$g_{1,1,T}$ provides a negative answer to
Question~$\ref{question:2}$.
\Part{(b)}
The families $g_{1,1,T}$ and $g_{1,2,T}$ satisfy
\begin{align*}
  \Ecal(g_{1,1,T}) \cap \Ecal(g_{1,2,T}) &= \Ecal(g_{1,2,T}) = \{1,3,5,7,\ldots\}, \\
  \Ecal(g_{1,1,T}) \trianglebin \Ecal(g_{1,2,T}) &= \Ecal(g_{1,1,T})\setminus\Ecal(g_{1,2,T}) = \{2,4,6,\ldots\}, 
\end{align*}
so they provide a negative answer to Question~$\ref{question:3}$.
\end{parts}
\end{proposition}

We note that one can use these families to construct negative answers
to Question~\ref{question:3} with sparser sets. For example,
$\Ecal(g_{2,0,T})=\{1,2,4,8,\ldots\}$ is an infinite, but
exponentially sparse, subset of $\Ecal(g_{1,1,1})$.

\section{Conjecture~\ref{conjecture:4} Implies Conjecture~\ref{conjecture:1}}
\label{section:conj4implyconj1}

In this section we sketch the proof, essentially due to
Xie~\cite[Theorem~4.3]{arxiv1106.1825}, that
Conjecture~\ref{conjecture:4} implies Conjecture~\ref{conjecture:1}.
More generally, we show that Conjecture~\ref{conjecture:4} implies a
generalizataion of Conjecture~\ref{conjecture:1} to families of
arbitrary dimension. 

\begin{theorem}
\label{theorem:thm4implythm1}
Assume that Conjecture~$\ref{conjecture:4}$ is true for a given $N\ge1$.
Let 
\[
  f_T:\PP^N_T\dashrightarrow\PP^N_T
\]
be a family of dominant rational maps over a smooth irreducible base
variety~$T$, all defined over an algebraically closed field~$K$.  Then
for all $\e>0$, the set
\[
  \bigl\{ t\in T(K) : \d(f_t) \le \d(f_T)-\e \bigr\}
\]
is contained in a proper Zariski closed subset of~$T$.
\end{theorem}

\begin{remark}
We thank Junyi Xie for pointing out that the proof
of~\cite[Theorem~4.3]{arxiv1106.1825} shows that
Conjecture~\ref{conjecture:4} implies that the map $t\to\d(f_t)$ is
lower semi-continuous, which strengthens the conclusion of
Theorem~\ref{theorem:thm4implythm1}.
\end{remark}

\begin{proof}[Proof of Theorem~$\ref{theorem:thm4implythm1}$]
We first view~$f=f_T$ as a rational map over the function
field~$\overline{K(T)}$. Then using the fact that~$\g_N>0$
and the definition of dynamical degree, we find that for any $k\ge1$ we have
\begin{align*}
  \lim_{n\to\infty} \left(\g_N \frac{\deg(f^{(k+1)n})}{\deg(f^{kn})}\right)^{1/n}
  &= \lim_{n\to\infty} \g_N^{1/n}  
     \frac{\bigl(\deg(f^{(k+1)n})^{1/(k+1)n}\bigr)^{k+1}}{\bigl(\deg(f^{kn})^{1/kn}\bigr)^k} \\
  &= \frac{\d(f)^{k+1}}{\d(f)^k} = \d(f).
\end{align*}
In particular, we can find an $m=m(\e,N)$ such that for all $0\le k<N$ we have
\begin{equation}
  \label{eqn:gdegf2mfm1mgedfe}
  \left(\g_N\cdot \frac{\deg(f^{(k+1)m})}{\deg(f^{km})}\right)^{1/m} \ge \d(f)- \e.
\end{equation}

We next observe that for any family $g:\PP^N_T\dashrightarrow\PP^N_T$
of dominant rational maps, the set
\[
  U(g):=\bigl\{t\in T(K) : \deg(g_t)=\deg(g)\bigr\}
\]
is a non-empty Zariski open subset of~$T$. We set
\[
  U_\e := \bigcap_{k=0}^N U(f^{km}) \subset T(K),
\]
where $m=m(\e,N)$ is as in~\eqref{eqn:gdegf2mfm1mgedfe}.

Finally, for $t\in U_\e$ we compute
\begin{align*}
  \d(f_t)
  &= \d(f_t^m)^{1/m} \quad
    \text{follows easily from definition of $\d$,} \\*
  &\ge \biggl(\g_N \cdot \min_{0\le k<N} \frac{\deg(f_t^{(k+1)m})}{\deg(f_t^{km})}\biggr)^{1/m}
    \text{Conjecture~\ref{conjecture:4} applied to $f_t^m$,} \\
  &= \biggl(\g_N \cdot \min_{0\le k<N} \frac{\deg(f^{(k+1)m})}{\deg(f^{km})}\biggr)^{1/m}
    \text{since $\displaystyle t\in U_\e := \bigcap_{0\le k<N}  U(f^{km})$,}\\*
  &\ge \d(f)-\e
    \quad\text{from \eqref{eqn:gdegf2mfm1mgedfe}.}
\end{align*}
This completes the proof of Theorem~\ref{theorem:thm4implythm1}.
\end{proof}

\section{A Dynamical Degree Estimate for Monomial Maps}
\label{section:monomialmaps}
As noted in the introduction, Xie~\cite{arxiv1106.1825} has shown that
there is a constant $\g>0$ such that
\[
  \d(f) \ge \g\cdot \frac{\deg(f^2)}{\deg(f)}
  \quad\text{for all birational maps $f:\PP^2\dashrightarrow\PP^2$.}
\]
In this section we prove an analogous result for dominant monomial
maps. We recall that a \emph{monomial map} is an endomorphism of the
torus~$\GG_m^N$, i.e., a map
\[
  \f_A : \GG_m^N\longrightarrow\GG_m^N
\]
of the form
\[
  \f_A(X_1,\ldots,X_N)=
   \left(
    X_1^{a_{11}}X_2^{a_{12}}\cdots X_N^{a_{1N}},\ldots,
    X_1^{a_{N1}}X_2^{a_{N2}}\cdots X_N^{a_{NN}} \right),
\]
where~$A=(a_{ij})\in\Mat_N(\ZZ)$ is an $N$-by-$N$ matrix with integer
coefficients.  The associated rational
map~$\f_A:\PP^N\dashrightarrow\PP^N$ is dominant if and only if
$\det(A)\ne0$.

\begin{corollary}
\label{corollary:dfAgeminfAk}
Let~$A\in\Mat_N(\ZZ)$ be a matrix with $\det(A)\ne0$, and let
$\f_A:\PP^N\dashrightarrow\PP^N$ be the associated monomial map. Then
\[
  \d(\f_A) \ge \frac{2^{1/N}-1}{2N^2}
  \min_{0\le k\le N-1} \frac{ \deg(\f_A^{k+1}) }{ \deg(\f_A^k) }.
\]
\end{corollary}

Before starting the proof, we set some notation and quote a result due
to Hasselblatt and Propp.  For any set of real numbers~$S$, we let
$\maxplus(S)=\max(0,S)$, and for matrices~$A\in\Mat_N(\RR)$ with real
coefficients, we define
\[
  D(A) = \sum_{j=1}^N \maxplus\{-a_{ij} : 1\le i\le N\}
            + \maxplus\left\{\sum_{j=1}^N a_{ij} : 1\le i\le N\right\},
\]
and we write
\begin{align*}
  \|A\| &= \max\bigl\{ |a_{ij}| : 1\le i,j\le N \bigr\}, \\*
  \l(A) &= \max\bigl\{ |\a| : \text{$\a\in\CC$ is an eigenvalue of $A$} \bigr\},
\end{align*}
for the sup-norm and the spectral radius of the matrix~$A$.

\begin{proposition}[Hasselblatt--Propp]
\label{proposition:HPdyndeg}
Let $A\in\Mat_N(\ZZ)$ be a matrix satisfying $\det(A)\ne0$.
\begin{parts}
\Part{(a)}
The degree of $\f_A$ is given by~$\deg\f_A=D(A)$.  
\Part{(b)}
The dynamical degree of~$\f_A$ is given by $\d(\f_A)=\l(A)$.
\end{parts}
\end{proposition}
\begin{proof}
For~(a), see~\cite[Proposition~2.14]{MR2358970}, and for~(b),
see~\cite[Theorem~6.2]{MR2358970}.
\end{proof}  

We now verify that~$D(A)$ induces a distance function on~$\Mat_N(\RR)$
that is equivalent to the easier-to-deal-with sup norm.

\begin{lemma}
\label{lemma:DAeqinfnorm}
Let $A\in\Mat_N(\RR)$. Then
\[
  \frac{1}{2N}D(A) \le \|A\| \le N D(A).
\]
\end{lemma}
\begin{proof}
From the definition of $D(A)$ we see that for every~$i,j$ we have
\[
  -a_{ij} \le D(A)  \quad\text{and}\quad \sum_{k=1}^N a_{ik} \le D(A).
\]
Hence every $a_{ij}\ge -D(A)$, while for every $i,j$ we can estimate
\[
  a_{ij} = \sum_{k=1}^N a_{ik} - \sum_{\substack{k=1\\k\ne j\\}} a_{ik}
  \le D(A) + \sum_{\substack{k=1\\k\ne j\\}} D(A) = ND(A).
\]
This proves that $-D(A)\le a_{ij}\le ND(A)$, so in particular
$|a_{ij}|\le ND(A)$, so $\|A\|\le ND(A)$. And for the other
direction, the triangle inequality applied to each term in the
definition of $D(A)$ immediately gives $D(A)\le 2N\|A\|$.
\end{proof}

\begin{proposition}
\label{proposition:Ak1lelAAk2N}
Let $A\in\Mat_N(\RR)$ be a matrix. Then
\[
  \|A^{k+1}\| \le  \frac{\l(A)\cdot \|A^{k}\|}{2^{1/N}-1}
  \quad\text{for some $0\le k\le N-1$.}
\]
\end{proposition}

\begin{proof}
Write the characteristic polynomial of $A$ as
\[
\det(xI-A) = \prod_{i=1}^N (x-\l_i) = \sum_{j=0}^N (-1)^j\s_j x^{N-j},
\]
where~$\s_j$ is the $j$'th elementary symmetric polynomial
of~$\l_1,\ldots,\l_N$.  We note for future reference that~$\s_j$ is a
sum of~$\binom{N}{j}$ monomials, each monomial being a product of~$j$
of the~$\l_i$'s, which combined with $\l(A)=\max|\l_j|$ gives the
upper bound
\begin{equation}
  \label{eqn:sjleNjAj}
  |\s_j| \le \binom{N}{j}\l(A)^j.
\end{equation}
We are going to use the Cayley-Hamilton theorem, which says that~$A$
satisfies its characteristic polynomial.

For notational convenience, we set $\e:=2^{1/N}-1$. We suppose that
\begin{equation}
  \label{eqn:CkeCk1x}
  \|A^{k+1}\| >  \e^{-1} \l(A)  \|A^k\|\quad\text{for all $0\le k\le N-1$}
\end{equation}
and derive a contradiction.  Iterating~\eqref{eqn:CkeCk1x}, we find
that
\begin{equation}
  \label{eqn:CkeNkCNx}
  \l(A)^{N-k} \|A^k\| < \e^{N-k}\|A^N\|  \quad\text{for all $0\le k\le N-1$.}
\end{equation}
We use this to estimate
\begin{align*}
  \|A^N\|
  &= \left\| \sum_{j=1}^N (-1)^j\s_j A^{N-j} \right\|
  &&\text{Cayley-Hamilton theorem,} \\
  &\le \sum_{j=1}^N |\s_j| \|A^{N-j}\|
  &&\text{triangle inequality,} \\
  &\le \sum_{j=1}^N \binom{N}{j}\l(A)^j\|A^{N-j}\|
  &&\text{from \eqref{eqn:sjleNjAj},} \\
  &< \sum_{j=1}^N \binom{N}{j} \e^j \|A^N\|
  &&\text{from \eqref{eqn:CkeNkCNx} with $k=N-j$,} \\
  &= \bigl( (1+\e)^N - 1 \bigr) \|A^N\|
  &&\text{binomial theorem,} \\
  &= \|A^N\|
  &&\text{since $\e:=2^{1/N}-1$.}
\end{align*}
This strict inequality is a contradiction, so~\eqref{eqn:CkeCk1x} is false,
which completes the proof of Proposition~\ref{proposition:Ak1lelAAk2N}.
\end{proof}

\begin{proof}[Proof of Corollary $\ref{corollary:dfAgeminfAk}$]
We choose some $0\le k\le N-1$ such that
Proposition~\ref{proposition:Ak1lelAAk2N} holds, and then we use
Lemma~\ref{lemma:DAeqinfnorm} and
Proposition~\ref{proposition:HPdyndeg} to estimate
\[
  \frac{1}{2^{1/N}-1}
  \ge \frac{ \|A^{k+1}\| }{\l(A)\cdot \|A^{k}\|}
  \ge \frac{ (2N)^{-1} D(A^{k+1}) } { \l(A) N D(A^k) }
  = \frac{ \deg(\f_A^{k+1}) } { 2N^2 \d(\f_A) \deg(\f_A^k) }.
\]
This completes the proof of Corollary~\ref{corollary:dfAgeminfAk}.
\end{proof}  

\begin{remark}
We close with an observation.  Let $\f_A$ be a monomial map that is
birational, i.e., whose matrix satisfies $\det(A)=1$.  If we label the
eigenvalues of~$A$ so that $|\l_1|\ge|\l_2|\ge\cdots\ge|\l_N|$, then
the largest eigenvalue of $A^{-1}$ is $\l_N^{-1}$.
Proposition~\ref{proposition:HPdyndeg}(b) gives
\[
  \d(f^{-1}) = |\l_N^{-1}| = |\l_1\l_2\cdots\l_{N-1}| \le |\l_1|^{N-1} = \d(f)^{N-1}.
\]
It turns out that this holds for all birational maps, even at the
degree stage, before taking the dynamical degree limit. We thank
Mattias Jonsson for showing us the following proof.
\end{remark}

\begin{proposition}
\label{proposition:degfinvledegfn-1}
Let $f : P^N \dashrightarrow P^N$ be a birational map. Then
\begin{equation}
  \label{eqn:degf-1degfN-1}
  \deg(f^{-1}) \le  (\deg f)^{N-1}.
\end{equation}
\end{proposition}
\begin{proof}
\textup{(Mattias Jonsson, private communication)}
Blow up~$\PP^N$ to get a birational morphism $\pi: X\to\PP^N$ so that
$f$ lifts to a birational morphism $g:X\to\PP^N$;
see~\cite[Example~II.7.17.3]{hartshorne}.
Let $H\in\Div(\PP^N)$ be a hyperplane,
and set $D_g:=g^*H$ and $D_\pi:=\pi^*H$.
For $0\le k\le N$, let~$d_k$ be the intersection index
\[
  d_k := D_g^k \cdot D_\pi^{N-k}.
\]
The divisors $D_g$ and $D_\pi$ are big and nef, so the
Khovanskii--Teissier inequality~\cite[Corollary~1.6.3,
  Example~1.6.4]{MR2095471} tell us that the map $k\to \log d_k$ is
concave.  Since~$d_0=1$, this gives, $d_{N-1}\le d_1^{N-1}$.  Further,
we have $d_1=\deg(f)$, while the fact that~$f$ is birational implies
that $d_{N-1}=\deg(f^{-1})$.  This gives the desired result.
\end{proof}




\def\cprime{$'$}

\end{document}